\documentclass[12pt]{amsart}

\usepackage{cases}
\usepackage[english]{babel}
\usepackage{times,bm,amsfonts,amsmath,amssymb,dsfont} 
\usepackage{graphicx} 
\usepackage[babel=true]{csquotes}
\usepackage{mathabx}
\usepackage{pdfsync}

\makeatletter
\def\section{\@startsection{section}{1}\z@{.9\linespacing\@plus\linespacing}%
  {.7\linespacing} {\fontsize{13}{15}\selectfont\scshape\centering}}
\def\paragraph{\@startsection{paragraph}{4}%
  \z@{0.3em}{-.5em}%
  {$\bullet$ \ \normalfont\itshape}}
\makeatother

\newtheorem{theo}{Theorem}[section]
\newtheorem{prop}[theo]{Proposition}
\newtheorem{lem}[theo]{Lemma}

\newtheorem{conj}[theo]{Conjecture}

\theoremstyle{definition}

\theoremstyle{remark}
\newtheorem{rem}[theo]{Remark}
\newcommand\got[1]{{\bm{\mathfrak{#1}}}}
\makeatletter

\@addtoreset{equation}{section}
\makeatother

\usepackage{color}

\definecolor{gr}{rgb}   {0.,   0.69,   0.23 }
\definecolor{bl}{rgb}   {0.,   0.5,   1. }
\definecolor{mg}{rgb}   {0.85,  0.,    0.85}
\definecolor{yl}{rgb}   {0.8,  0.7,   0.}

\definecolor{webred}{rgb}{0.75,0,0}
\definecolor{webgreen}{rgb}{0,0.75,0}
\usepackage[citecolor=webgreen,colorlinks=true,linkcolor=webred]{hyperref}

\renewcommand{\d}{\, {\rm d}}

\newcommand{\Z}{\mathbb{Z}}
\newcommand{\N}{\mathbb{N}}
\newcommand{\R}{\mathbb{R}}

\newcommand{\Rp}{\R_{+}}

\newcommand{\spectre}{\lambda}

\newcommand{\spec}{\got{S}}

\newcommand{\dom}{\operatorname{Dom}}

\newcommand{\curl}{\operatorname{curl}}

\newcommand{\gq}{\got{q}}

\newcommand{\cF}{{\mathcal{F}}}

\newcommand{\bB}{{\bf B}}
\newcommand{\bA}{{\bf A}}
\newcommand{\sign}{\rm sign}

\newcommand{\dgG}{\got{\ell}}
\newcommand{\dg}{\got{h}}
\newcommand{\dgN}{\got{h}^{ \rm N}}
\newcommand{\dgD}{\got{h}^{ \rm D}}

\newcommand{\mudg}{\mu}
\newcommand{\mudgN}{\mu^{\rm N}}
\newcommand{\mudgD}{\mu^{\rm D}}

\newcommand\Wedge{\mathcal{W}}

\newcommand{\Lunid}{\got{g}}
\newcommand{\LunidT}{\hat{\Lunid}}

\newcommand{\Acyl}{{\bf \mathcal{A}}}
\newcommand{\Bcyl}{{\bf \mathcal{B}}}

\title[Lowest energy of a magnetic hamiltonian with an axisymmetric potential]{On the lowest energy of a 3D magnetic hamiltonian with axisymmetric potential}
\author{Nicolas Popoff}
\address{Laboratoire IRMAR, UMR 6625 du CNRS, Campus de Beaulieu, 35042 Rennes cedex, France
\\
 {\it E-mail address:} nicolas.popoff@univ-rennes1.fr}

\date{\today}

\begin{document}
\begin{abstract}
We study the bottom of the spectrum of a magnetic hamiltonian with axisymmetric potential in $\R^3$. The associated magnetic field is planar, unitary and non-constant. The problem reduces to a 1D family of singular Sturm-Liouville operators on the half-line. We study the associated band functions and we compare it to the ``de Gennes" operators arising in the study of a 2D hamiltonian with monodimensional, odd and discontinuous magnetic field. We show in particular that the lowest energy is higher in dimension 3.
\end{abstract} 
\maketitle

\section{Introduction}

\subsection{Description of the 2D model}
The action of a spatially inhomogeneous magnetic field on a two-dimensio\-nal electron gas has been the focus of several researches in the past decades (see \cite{PeetVas93} and \cite{Nal} for example). Indeed a perpendicular magnetic field (modeled be a 2D scalar vector field) modifies the transport properties of the electron gas (see \cite{PeetMat93}). 
The variation of the magnetic field can induce a quantum transport called ``edge current" (see \cite{HorSmi02} for a physical overview). The case of a unidimensional and non-decreasing magnetic field is described by the Iwatsuka model (see \cite{Iwa85} for the original article and \cite{DomGerRai} for more recent results). The perturbation by a unidimensional non-decreasing electric potential is studied in \cite{BruMirRai11} and \cite{dBiePu99} for example. These studies on elementary geometries help to understand the quantum hall effect.

In \cite{PeetRej}, the authors study among others the case of the magnetic field $B_{0}$ defined on $\R^2$ by $B_{0}(x,y)=\sign(x)$. Let $\bA_{0}(x,y):=|x|$ be a magnetic potential satisfying $\curl \bA_{0}=B_{0}$ and 
\begin{equation}
\label{D:hamiltonien2d}
H_{0}:=(-i\nabla-\bA_{0})^2=D_{x}^2+(D_{y}-|x|)^2 \ , \quad (x,y)\in \R^2
\end{equation}
the associated hamiltonian with $D=-i\partial$. In \cite{PeetRej} the formal spectral analysis of the hamiltonian $H_{0}$ brings the transport properties of a 2D electron gas submitted to the magnetic field $B_{0}$ along the singularity of the magnetic field. Physical arguments show that the classical trajectories correspond to the so-called {\it snake orbits} (see \cite{PeetRej02} and \cite{PeetRej}). Mathematical properties of the hamiltonian $H_{0}$ for small electric perturbations are studied in \cite{HisSoc} (see also \cite{BruMirRaik13} for a related hamiltonian on a half-plane).  

We denote by $\spec(\mathcal{L})$ the spectrum of a self-adjoint operator $\mathcal{L}$. Let 
 \begin{equation}
\label{D:theta0}
\Theta_{0}:=\inf \spec(H_{0})
\end{equation}
 be the bottom of the spectrum of the operator $H_{0}$. This spectral quantity has been introduced in \cite{SaGe63} for a problem coming from the modeling of the phenomenon of ``surface superconductivity" (see section \ref{S:superconductivity}). The study of the spectrum of $H_{0}$ leads to the 1D parameter family of operators pencil 
\begin{equation}
\label{D:dgtau}
\dg_{0}(\tau):=-\partial_{x}^2+(|x|-\tau)^2, \quad x\in \R
\end{equation}
 where $\tau \in \R$ is the Fourier variable dual to $y$. These operators are known as the {\it de Gennes} operators (see Subsection \ref{SS:deGennes} for references and former results). The eigenvalues of such an operator family seen as functions of $\tau$ are often called ``band functions" or ``dispersion curves". Their analysis brings the spectral properties of the hamiltonian $H_{0}$.

\subsection{Problematic and description of the 3D model}
The aim of this article is to study the bottom of the spectrum of a hamiltonian associated to a particular planar inhomogeneous magnetic field of $\R^3$ whose associated magnetic potential is axisymmetric. Let us denote by $(r,\theta,z)$ the cylindrical coordinates of $\R^3$. In the case where the magnetic potential has the shape $\bA(r,\theta,z)=(0,0,a(r))$, the associated magnetic field is planar and writes $\bB(r,\theta,z)=b(r)(-\sin\theta,\cos\theta,0)$ with $b(r)=a'(r)$. Its field lines are circles centered at the origin. Under general assumptions on the function $b$ the classical trajectories of a particle in such a magnetic field are described in \cite[Section 4]{Yaf03}. The spectrum and the scattering properties of the hamiltonian $H_{\bA}:=(-i\nabla-\bA)^2$ associated to such magnetic fields are studied in \cite{Yaf03} and \cite{Yaf08}. The particular case of a magnetic field created by an infinite rectilinear current in the $z$ direction is studied in \cite{Yaf03}: in that case $b(r)=r^{-1}$. The spectrum of $H_{\bA}$ is the half-line $\R_{+}$ and the band functions are decreasing from $+\infty$ to 0. In \cite{Yaf08}, more general magnetic hamiltonians with axisymmetric potentials are described and the author gives conditions for the spectrum of $H_{\bA}$ to be the half-line $\R_{+}$. In \cite[Section 4]{Yaf08}, the particular case $b(r)=1$ is treated. The author shows that the band functions associated to axisymmetric functions of $\R^3$ loose their monotonicities and he deduces that the bottom of the spectrum of $H_{\bA}$ is positive. In this article we study in details the bottom of the spectrum of the magnetic hamiltonian for the case $b(r)=1$ and we make a comparison with the 2D hamiltonian defined in \eqref{D:hamiltonien2d}. We introduce a new operator pencil that can be seen as a 2D version of the de Gennes operator defined in \eqref{D:dgtau}.

We present here the magnetic hamiltonian for the case $b(r)=1$ and the associated magnetic potential $a(r)=r$. In the cartesian coordinates of $\R^3$ the magnetic potential writes $$\Acyl(x,y,z):=(0,0,\sqrt{x^2+y^2}) \ . $$ The associated magnetic field $\Bcyl:=\curl \Acyl$ satisfies $\Bcyl(x,y,z)=(\sin\theta,-\cos\theta,0)$. This magnetic field is unitary and non-constant. The restriction of the magnetic field $\Bcyl$ to a plane of the form $\{y=ax\}$ with $a\in\R$ has the shape of the magnetic field $B_{0}$ associated to the hamiltonian \eqref{D:hamiltonien2d}. Let
\begin{equation}
\label{D:Aaxisym}
H_{\Acyl}:=(-i\nabla-\Acyl)^2=D_{x}^2+D_{y}^2+(D_{z}-\sqrt{x^2+y^2})^2
\end{equation}
be the hamiltonian associated to the magnetic field $\Bcyl$ acting on $L^2(\R^3)$ and 
\begin{equation}
\label{}
\Xi_{0}:=\inf \spec\left(H_{\Acyl} \right)
\end{equation}
its lowest energy. We know from \cite{Yaf08} that $\spec(H_{\Acyl})=[\Xi_{0},+\infty)$. One of our goals is to compare $\Xi_{0}$ with the lowest energy $\Theta_{0}$ of the hamiltonian $H_{0}$ defined in \eqref{D:hamiltonien2d}. Let $\cF_{z}$ be the partial Fourier transform in the $z$-variable. We have the direct integral decomposition (see \cite{ReSi78}):
\begin{equation}
\label{E:Fourier2}
\cF_{z}^{*}H_{\Acyl}\cF_{z} =\int_{\tau\in\R}^{\bigoplus}  \dgG(\tau)\d \tau
\end{equation}
with 
\begin{equation}
\label{D:dgG}
\dgG(\tau):=-\Delta_{x,y}+(\|(x,y)\|-\tau)^2, \quad (x,y)\in \R^2 
\end{equation}
where $\|\cdot\|$ denotes the euclidean norm of $\R^2$. The operator $\dgG(\tau)$ has compact resolvent and we denote by $\lambda_{1}(\tau)$ its first eigenvalue. This operator can be seen as a 2D version of the operator $\dg_{0}(\tau)$ arising in the study of $H_{0}$. Using \eqref{E:Fourier2} we have the fundamental relation
$$\Xi_{0}=\inf_{\tau\in\R} \lambda_{1}(\tau) \ . $$
The restriction of $\dgG(\tau)$ to axisymmetric functions reduces the problem to the singular 1D operator $\Lunid(\tau)$ introduced in Subsection \ref{SS:reductionto1D}. We denote by $(\zeta_{n}(\tau))_{n\in \N^{*}}$ the spectrum of this operator and we have $\zeta_{1}(\tau)=\lambda_{1}(\tau)$. We are interested in the description of the first band function $\zeta_{1}$ and of its infimum $\Xi_{0}$. We are also interested in comparisons with the spectral quantities associated to the operator $\dg_{0}(\tau)$ coming from the study of the hamiltonian $H_{0}$ defined in \eqref{D:hamiltonien2d}.

\subsection{Connection with superconductivity}
\label{S:superconductivity}
The modelization of the superconductivity phenomenon leads to study the minimizers of the Gainzburg-Landau functional. For a strong external magnetic field the superconductivity phenomenon is destroyed. The linearization of the Gainzburg-Landau functional in that case leads to study the magnetic Laplacian with the natural Neumann boundary conditions  (see \cite{GiPh99}). This operator is denoted by $H_{\bA, \, \Omega}$ where $\bA$ is the magnetic potential and $\Omega\subset \R^3$ is the domain. The bottom of its spectrum is denoted by $\spectre(\bB, \Omega)$ since it depends only of the magnetic field $\bB:=\curl \bA$. The critical value of the magnetic field for which the superconductivity disappears in a type II superconductor $\Omega$ is linked to $\spectre(\bB, \Omega)$ (see \cite{FoHe3} for example and \cite{FouHel10} for more references). This gives an important motivation for the comprehension of the behavior of $\spectre(\bB, \Omega)$ for large values of $\bB$. For $x\in \overline{\Omega}$ we denote by $\Pi_{x}$ the tangent cone to $\Omega$ at the point $x$ and $\widetilde{\bB}_{x}:=|\bB(x)|^{-1}\bB(x)$ the normalized magnetic field frozen at $x$. Ones should expect that $\lambda(\bB,\Omega)$ behaves like $\inf_{x\in \overline{\Omega}}|\bB(x)| \lambda(\widetilde{\bB}_{x},\Pi_{x})$ for large magnetic fields, indeed all the known asymptotics of $\lambda(\bB,\Omega)$ have shown this structure. An ongoing work (\cite{BoDauPop13B}) is in progress to get the asymptotics for general corner domains and non-vanishing regular magnetic fields. 

In the perspective to determine the asymptotics of $\lambda(\bB,\Omega)$ for large magnetic fields, it is crucial to have comparisons between all the possible values of $\lambda(\widetilde{\bB}_{x},\Pi_{x})$ for $x\in \overline{\Omega}$. When the boundary of $\Omega$ is regular, the tangent cones $\Pi_{x}$ are either spaces or half-spaces. The spectral model quantity $\lambda(\widetilde{\bB}_{x},\Pi_{x})$ is minimal and equal to $\theta_{0}$ in the case where $\Pi_{x}$ is a half-space with the magnetic field tangent to the boundary (see \cite{LuPan00}). In the case where the boundary of the domain has an edge of opening $\alpha$, it is necessary to study the Neumann magnetic Laplacian on a new model domain: the infinite wedge of opening $\alpha$ denoted by $\Wedge_{\alpha}$. First studies of this operator are presented in \cite{Pan02}, \cite{Bon06} and \cite{Pof13T} for particular geometries. Let $\bB$ be a constant unitary magnetic field. We denote by $b_{\perp}$ the component of $\bB$ orthogonal to the plane of symmetry of the wedge. If $b_{\perp}\neq0$, the magnetic Laplacian $b_{\perp}^{-1}H_{\bA, \, \Wedge_{\alpha}}$ degenerates formally toward the operator $H_{\Acyl}$ (defined in \eqref{D:Aaxisym}) when the opening angle $\alpha$ goes to 0. A formal analysis and several numerical computations show that $\spectre(\bB,\Wedge_{\alpha})$ seems to converge to $b_{\perp}\Xi_{0}$ when the opening angle $\alpha$ goes to 0 (see \cite[Chapter 6]{Popoff}). Therefore the comparison between $b_{\perp}\Xi_{0}$ and the spectral model quantities associated to the points of the regular boundary of $\Omega$ will brings the asymptotics of the first eigenvalue of the magnetic Laplacian on a domain with an edge of small opening. In this article we prove $\Theta_{0}<\Xi_{0}$. An application of the comparison between regular and singular model problems can be found in \cite{PopRay12}. The semi-classical Laplacian with a constant magnetic field in a domain with a curved edge (a lens) is studied. The authors make an assumption on a 2D band function related to the conjecture \ref{C:} and use the tools of the semi-classical analysis to provide complete expansion of the eigenvalues of the magnetic Laplacian on the lens.

\subsection{Contents and main results}
 In Section 2 we reduce the problem to a family of singular 1D Sturm-Liouville operators $(\Lunid(\tau))_{\tau\in\R}$ on the half-line. We study the eigenvalues $\zeta_{n}(\tau)$ of this 1D operators: we give a two-terms asymptotics when the Fourier parameter gets large and we provide an upper bound for the minimum $\Xi_{0}$. In Section 3 we give a formula for the derivative of $\zeta_{n}(\tau)$ with respect to $\tau$ and we use it to show that $\Theta_{0}<\Xi_{0}$. We also give a criterion to characterize the minima of $\zeta_{1}(\tau)$. In Annex \ref{S:num} we give numerical computations of $\zeta_{1}(\tau)$.

\section{Description of the 1D operators}
\subsection{The de Gennes operator}
\label{SS:deGennes}
We first recall known results about the de Gennes operator arising in the study of the hamiltonian $H_{0}$ defined in \eqref{D:hamiltonien2d}. Let $\cF_{y}$ be the partial Fourier transform in the $y$-variable. We have the following direct integral decomposition 
\begin{equation}
\label{FdecdG}
\cF_{y}^{*} H_{0}\cF_{y}:=\int_{\tau\in \R}^{\bigoplus}\dg_{0}(\tau) \d \tau  
\end{equation}
where $\dg_{0}(\tau)$ is defined in \eqref{D:dgtau}. For all $\tau \in \R$ the operator $\dg_{0}(\tau)$ has compact resolvent and we denote by $\mudg_{n}(\tau)$ its $n$-th eigenvalue. Let $\dgN_{0}(\tau)$ (resp. $\dgD_{0}(\tau)$) be the operator $\partial_{x}^2+(x-\tau)^2$ acting on $L^2(\R_{+})$ with Neumann (resp. Dirichlet) boundary condition in $x=0$. We denote by $\mudgN_{n}(\tau)$ (resp. $\mudgD_{n}(\tau)$) its $n$-eigenvalue. We have (see \cite{FouHel10}) for all $n\geq1$ that $\mudg_{2n-1}(\tau)=\mudgN_{n}(\tau)$ and $\mudg_{2n}(\tau)=\mudgD_{n}(\tau)$. When $\tau$ goes to $+\infty$, $\mudgN_{n}(\tau)$ is exponentially close to the Landau level $2n-1$ (see \cite{FouHePer11}):
\begin{equation}
\label{E:limtaugranddG}
\forall n\geq1, \exists C>0, \exists \tau_{0}, \forall \tau \geq \tau_{0}, \quad |\mudgN_{n}(\tau)-(2n-1)| \leq C e^{-\tau^2/2} 
\end{equation}
A two-terms asympotics is computed formally in \cite{PeetRej}. The more precise following expansions are rigorously proved in \cite[Section 1.5]{Popoff}: 
\begin{equation}
\label{dgtunD}
\mudgD_{n}(\tau)=2n-1+\frac{2^{n}}{(n-1)!\sqrt{\pi}}\tau^{2n-1}e^{-\tau^2}\left(1-\frac{n^2-n+1}{2\tau^2}+O\left(\frac{1}{\tau^4}\right)\right) 
\end{equation}
and
\begin{equation}
\label{dgtunN}
\mudgN_{n}(\tau)=2n-1-\frac{2^{n}}{(n-1)!\sqrt{\pi}}\tau^{2n-1}e^{-\tau^2}\left(1-\frac{n^2-n-1}{2\tau^2} +O\left(\frac{1}{\tau^4}\right)\right) \ . 
\end{equation}
Let $u_{n,\tau}$ be a normalized eigenfunction of $\dgN_{0}(\tau)$ associated to $\mudgN_{n}(\tau)$. Using the technics from \cite{DauHe93} and \cite{BolHe93}, it is known that 
\begin{equation}
\label{F:DH93}
(\mudgN_{n})'(\tau)=(\tau^2-\mudgN_{n}(\tau))u_{n,\tau}^2(0)
\end{equation}
and that there exists $\xi_{0}^{n}\in \R$ such that $\tau \mapsto \mudgN_{n}(\tau)$ is decreasing on $(-\infty,\xi_{0}^{n})$ and increasing on $(\xi_{0}^{n},+\infty)$. Moreover the unique minimum of $\mudgN_{n}$ is non-degenerate and we have $\Theta_{0}=\inf_{\tau}\mudgN_{1}(\tau)$. If we denote by $\xi_{0}:=\xi_{0}^{1}$, \eqref{F:DH93} provides $\xi_{0}^2=\Theta_{0}$. Using \eqref{FdecdG} we get $\inf \spec(H_{0})=\Theta_{0}$. Numerical computations (see \cite{SaGe63}, \cite{Cha94} or \cite{Bo08} for a more rigorous analysis) show that $(\xi_{0},\Theta_{0})\approx (0.7682,0.5901)$.

\subsection{Reduction to a 1D problem}
\label{SS:reductionto1D}
We reduce the study of the first eigenvalue of $\dgG(\tau)$ to a 1D singular Sturm-Liouville operator on a weighted space. In the polar coordinate $(r,\phi)$ the operator $\dgG(\tau)$ defined in \eqref{D:dgG} writes 
$$-\partial_{r}^2-\frac{1}{r}\partial_{r}-\frac{1}{r^2}\partial_{\phi}^2+(r-\tau)^2 \ , \quad (r,\phi)\in \R_{+}\times (-\pi,\pi) \ . $$
The eigenfunctions associated to the first eigenvalue are axisymmetric therefore we restrict the operator to the functions that does not depend on the variable $\phi$. In other words we restrict our study to the case where the magnetic quantum number $m$ is equal to 0.

 Let $L^2_{r}(\R_{+})$ be the space of the functions squared integrable on the half axis $\R_{+}$ for the weight $r \d r$.We denote by
$$\langle u,v \rangle_{L^2_{r}(\R_{+})}:= \int_{\R_{+}}u(r) v(r) r \d r $$
the scalar product associated to $L^2_{r}(\R_{+})$. Let $B_{r}^1(\R_{+}):=\{ u\in L^2_{r}, u'\in L^2_{r}(\R_{+}), ru\in L^2_{r}(\R_{+}) \} \ . $ We define the operator 
$$\Lunid(\tau)=-\partial_{r}^2-\frac{1}{r}\partial_{r}+(r-\tau)^2$$
on the domain
$$\dom(\Lunid(\tau))=\{u\in B^1_{r}(\R_{+}), u''\in L^2_{r}(\R_{+}), \frac{1}{r}u' \in L^2_{r}(\R_{+}), r^2u \in L^2_{r}(\R_{+}), (ru'(r))_{|r=0}=0 \} \ . $$
The operator $\Lunid(\tau)$ is unitary equivalent to the restriction of $\dgG(\tau)$ to axisymmetric functions of $\R^2$. The form domain of $\Lunid(\tau)$ is $B^1_{r}(\R_{+})$ and the associated quadratic form is
$$\gq_{\tau}(u):=\int_{\R_{+}}\left(|u'(r)|^2+(r-\tau)^2|u(r)|^2 \right) r \d r \ .$$

Using the results from \cite{BolCa72} we get that the operator $\Lunid(\tau)$ has compact resolvent and we denote by $\zeta_{n}(\tau)$ its $n$-th eigenvalue. We have $\zeta_{1}(\tau)=\lambda_{1}(\tau)$ and
\begin{equation}
\label{E:relationfond}
\Xi_{0}=\inf_{\tau\in \R} \zeta_{1}(\tau) \ ,
\end{equation}
moreover the $(\zeta_{n}(\tau))_{n \geq 0}$ are the eigenvalues of $\dgG(\tau)$ which have axisymmetric eigenfunctions.

\subsection{Elementary results about the spectrum of the 1D operator}
The boundary value problem associated to the eigenvalue $\zeta_{n}(\tau)$ is the following: 
\begin{subnumcases}{}
\label{equation-de-Fuchs}
-r u''(r)-u'(r)+r(r-\tau)^2u(r)=r\zeta_{n}(\tau)u(r) \ , \quad r>0 \ , 
\\
\label{condition-aux-limites-Fuchs}
ru'(r)_{|r=0}=0 \ .
\end{subnumcases}
The EDO \eqref{equation-de-Fuchs} is singular and the point $r=0$ is regular singular, therefore using Fuchs theory for EDO's and the boundary condition \eqref{condition-aux-limites-Fuchs} we get (see \cite[Proposition 6.4]{Popoff} for more details and \cite[Lemma 4.1]{Yaf08} for a more general case):
\begin{prop}
\label{P:Fuchs}
For all $\tau\in \R$ and for all $n \geq 1$, $\zeta_{n}(\tau)$ is a simple eigenvalue of $\Lunid(\tau)$. We denote by $z_{n,\tau}$ an associated eigenfunction. The function $z_{n,\tau}$ is the restriction on $\R_{+}$ of an analytic function on $\R$, moreover it satisfies the Neumann boundary condition 
\begin{equation}
\label{E:vpneumann}
z_{n,\tau}'(0)=0 \ .
\end{equation}
\end{prop}
Since the form domain of $\Lunid(\tau)$ does not depend on $\tau$, we deduce from Kato's theory (see \cite{Kato}) that all the $\zeta_{n}(\tau)$ are analytic with respect to $\tau$. Moreover the Feynman-Hellmann (\cite{IsZh88}) formula provides
\begin{equation}
\label{E:FH}
\forall n \geq 1, \forall \tau \in \R, \quad \zeta_{n}'(\tau)=-2 \int_{\R_{+}}(r-\tau)^2z_{n,\tau}(r) r \d r \ , 
\end{equation}
 and we deduce that for all $n \geq1$, the function $\tau\mapsto \zeta_{n}(\tau)$ is decreasing on $(-\infty,0)$. 

In the special case $\tau=0$, the operator $\Lunid(\tau)$ is also known as the Laguerre operator (see \cite{AbSt64}) whose eigenvalues are known: 
$$\forall n \geq 1, \quad \zeta_{n}(0)=4n-2 \ . $$
\begin{rem}
The eigenvalues of  the 2D harmonic oscillator $\dgG(0)=-\Delta+\|(x,y)\|^2$ with $(x,y)\in \R^2$ are the even positive  integers. Only the eigenspaces associated to eigenvalues of the form $4n-2$ have axisymmetric eigenfunctions. 
\end{rem}
\subsection{Limits for large parameters}
Using the lower bound $(r-\tau)^2\geq\tau^2$ for $\tau \leq0$ we deduce from the min-max principle that for all $\tau \leq 0$ we have $\zeta_{n}(\tau) \geq \tau^2$ and therefore 
\begin{equation}
\label{E:limtau-inf}
\lim_{\tau\to-\infty} \zeta_{n}(\tau)= +\infty \ .
\end{equation}
In the case of the de Gennes operator $\dg_{0}(\tau)$ the eigenfunctions concentrate for large $\tau$ in the wells of the potential $(r-\tau)^2$ and therefore the eigenvalues of $\dg_{0}(\tau)$ converge toward the Landau level for large $\tau$ (see \eqref{E:limtaugranddG}). This is again true for the eigenvalues of the operator $\Lunid(\tau)$, indeed the potential satisfies the hypothesis of \cite[Proposition 3.6]{Yaf08} and we deduce: 
\begin{prop}
\label{P:limite}
We have 
$$\forall n \geq1, \quad \lim_{\tau\to+\infty}\zeta_{n}(\tau)=2n-1 \ .$$
\end{prop}
 However we work in a weighted space and the harmonic approximation that consists in using the Hermite's functions as quasi-modes is not as good as in the case of the de Gennes operator. It is proven in \cite[Proposition 4.7]{Yaf08} that
 \begin{equation}
\label{E:EstYafaev}
 \forall n\geq1, \quad \exists \gamma_{n}>0,\,  \exists \tau_{n}>0,\,  \forall \tau>\tau_{n}, \quad \zeta_{n}(\tau) \leq (2n-1)-\gamma_{n}\tau^{-2} \ . 
 \end{equation}
 We give a two-terms asymptotics of the band functions $\zeta_{n}$ for large $\tau$:
 \begin{prop}
 \label{P:asymptoticzeta}
 For all $n\geq1$ we have
 \begin{equation}
\label{E:DAzetan}
 \zeta_{n}(\tau)\underset{\tau\to+\infty}{=}2n-1-\frac{1}{4\tau^2}+O\left(\frac{1}{\tau^3}\right) \ . 
 \end{equation}
 \end{prop}
\begin{proof}
We use the change of variable $t=\frac{r^2}{2}$ and we get that the operator $\Lunid(\tau)$ is unitary equivalent to 
$$-2\partial_{t}t\partial_{t}+(\sqrt{2t}-\tau)^2, \, \quad t>0$$
defined on $\{u\in H^1(\R_{+}), \, t\, u\in H^1(\R_{+}) \}$. Let us remark that we are now working on an unweighted space. We center and rescale this operator with the change of variable defined by $x=\tau^{-1}(t-\frac{\tau^2}{2})$ and we are led to study the operator 
\begin{equation}
\label{D:Lunidsc}
\Lunid^{\rm sc}(\tau):=-\partial_{x}^2-2\tau^{-1}\partial_{x}x\partial_{x}+\tau(\sqrt{2x+\tau}-\sqrt{\tau})^2 \ , x\in J_{\tau}
\end{equation}
defined on $\dom(\Lunid^{\rm sc}(\tau))=\{u\in L^2(J_{\tau}), \, \partial_{x}^2u\in L^2(J_{\tau}), \,\partial_{x}x\partial_{x}u\in L^2(J_{\tau}), \, V_{\tau}u\in L^2(J_{\tau})  \}$ where $J_{\tau}:=(-\frac{\tau}{2},+\infty)$ and the normalized potential is $$V_{\tau}(x):=\tau(\sqrt{2x+\tau}-\sqrt{\tau})^2 \ . 
$$ The expansion of $V_{\tau}$ near 0 provides 
\begin{equation}
\label{E:DLV}
\exists x_{0}>0, \exists C>0, \, \forall x\in (-x_{0},x_{0}), \, \forall \tau\geq 1, \quad |V_{\tau}(x)-\left(x^2-\frac{x^3}{\tau}+\frac{5x^4}{4\tau^2}\right)|\leq C\frac{x^5}{\tau^3} \ . 
\end{equation}
Let $h=\tau^{-1}$. We define $H=H_{0}+hH_{1}+h^2H_{2}$ with 
\begin{equation}
\left\{
\begin{aligned}
&H_{0}:=-\partial_{x}^2+x^2\\
&H_{1}:=-2\partial_{x}x\partial_{x}-x^3\\
&H_{2}:=\tfrac{5}{4}x^4 
\end{aligned}
\right.
\end{equation}
acting on functions of $L^2(\R)$. Let us notice that the formal two-terms expansion of $\Lunid^{\rm sc}(\tau)$ for large $\tau$ corresponds to the operator $H$. In order to construct a quasi-mode for $\Lunid^{\rm sc}(\tau)$ when $\tau$ gets large we construct a quasi-mode for $H$ when $h$ gets small. We are looking for an approximate eigenpair $(E_{h},u_{h})$ for the operator $H$ with $E_{h}=E_{0}+hE_{1}+h^2E_{2}$ and $u_{h}=u_{0}+hu_{1}+h^2u_{2}$. Solving formally $Hu_{h}=E_{h}u_{h}$ leads to solve the following equations: 
\begin{subnumcases}{}
\label{Eqsyst-1}
H_{0}u_{0}=E_{0}u_{0}\\
\label{Eqsyst-2}
H_{1}u_{0}+H_{0}u_{1}=E_{0}u_{1}+E_{1}u_{0}\\
\label{Eqsyst-3}
H_{2}u_{0}+H_{1}u_{1}+H_{0}u_{2}=E_{2}u_{0}+E_{1}u_{1}+E_{0}u_{2}
\end{subnumcases}
We solve \eqref{Eqsyst-1} by taking $$E_{0}=E_{0,n}:=2n-1 \quad \mbox{and} \quad u_{0}=u_{0,n}:=\Psi_{n}, \quad n\geq1$$ where $\Psi_{n}$ denotes the $n$-th normalized Hermite's function with the convention that $\Psi_{1}(x):=\pi^{-1/4}e^{-x^2/2}$ is the first Hermite's function. We take the scalar product of \eqref{Eqsyst-2} against $u_{0,n}$ and we get $E_{1}=\langle H_{1}u_{0,n},u_{0,n}\rangle$. When $n$ is odd (respectively even), the $n$-th Hermite's function is even (respectively odd), $H_{1}u_{0,n}$ is odd (respectively even) and $u_{0,n}\cdot H_{1}u_{0,n}$ is odd. We deduce that 
$$E_{1}=0 \ . $$
We now find $u_{1}$: we have to solve 
\begin{equation}
\label{E:jenaimarre}
(H_{0}-E_{0})u_{1}=-H_{1}u_{0,n} \ . 
\end{equation}
We decompose $(-H_{1}u_{0,n})(x)=x^3\Psi_{n}(x)+2\Psi_{n}'(x)+\Psi_{n}''(x)$ along the basis of Hermite's functions: 
$$-H_{1}u_{0,n}=a_{n}\Psi_{n-3}+b_{n}\Psi_{n-1}+c_{n}\Psi_{n+1}+d_{n}\Psi_{n+3} \ . $$
Using that for $n\geq1$ (see \cite{AbSt64}):
$$x\Psi_{n}(x)= \sqrt{\frac{n-1}{2}}\Psi_{n-1}+\sqrt{\frac{n}{2}}\Psi_{n+1} \quad \mbox{and}\quad \Psi_{n}'(x)= \sqrt{\frac{n-1}{2}}\Psi_{n-1}-\sqrt{\frac{n}{2}}\Psi_{n+1}  \ ,$$
computations yield
\begin{equation}
\forall n\geq1, \, 
\left\{
\begin{aligned}
&a_{n}=3\cdot2^{-3/2}\sqrt{(n-1)(n-2)(n-3)}\\
&b_{n}=2^{-3/2}(n-1)\sqrt{n-1}\\
&c_{n}=2^{-3/2}n\sqrt{n}\\
&d_{n}= 3\cdot2^{-3/2}\sqrt{n(n+1)(n+2))} \ .
\end{aligned}
\right.
\end{equation}
For solving \eqref{E:jenaimarre} we take $$u_{1}=u_{1,n}:=-\frac{a_{n}}{6}\Psi_{n-3}-\frac{b_{n}}{2}\Psi_{n-1}+\frac{c_{n}}{2}\Psi_{n+1}+\frac{d_{n}}{6}\Psi_{n+3} \ . $$ 
We take the scalar product of \eqref{Eqsyst-3} with $u_{0,n}$ and we get 
$$E_{2}=E_{2,n}:=\langle H_{2}u_{0,n},u_{0,n}\rangle+\langle H_{1}u_{1,n},u_{0,n}\rangle \ . $$
We have $\langle H_{2}u_{0,n},u_{0,n}\rangle=\frac{5}{4}\|x^2\Psi_{n}\|^2_{L^2(\R)}=\frac{15}{16}(2n^2-2n+1)$ and 
\begin{align*}
\langle H_{1}u_{1,n},u_{0,n}\rangle=\langle u_{1,n},H_{1}u_{0,n}\rangle=\left(\frac{a_{n}^2}{6}+\frac{b_{n}^2}{2}-\frac{c_{n}^2}{2}-\frac{d_{n}^2}{6}\right)=-\frac{1}{16}\left(30n^2-30n+19 \right) \ .
\end{align*}
Therefore we deduce 
$$E_{2,n}=-\frac{1}{4} \ .$$
We now find $u_{2}$ by solving 
$$(H_{0}-E_{0})u_{2}=E_{2,n}u_{0,n}-H_{2}u_{0,n}-H_{1}u_{1,n} \ . $$
By construction of $E_{2,n}$ the right hand side is orthogonal to $u_{0,n}$. Thanks to the Fredholm alternative we get a unique solution $u_{2,n}$ orthogonal to $u_{0,n}$. Moreover since the right hand side has exponential decay at infinity, $u_{2,n}$ has also exponential decay at infinity.

We now evaluate the quasi-pair constructed. Let $E_{h,n}:=2n-1-\frac{1}{4}h^2$ be the approximate eigenvalue constructed and $u_{h,n}:=u_{0,n}+hu_{1,n}+h^2u_{2,n}$ be the associated approximate eigenfunction. Since $u_{h,n}$ has exponential decay we have for all $n\geq1$ that there exists $C_{n}>0$ such that for all $h\in (0,1)$:
\begin{equation}
\label{E:estimationuhn}
 \|Hu_{h,n}-E_{h,n}u_{h,n}\|_{L^2(\R)}^2\leq C_{n}h^3  \quad \mbox{and} \quad |\|u_{h,n}\|_{L^2(\R)}-1| \leq C_{n}h \ . 
 \end{equation}
Let $\chi_{\tau}\in \mathcal{C}_{0}^{\infty}(J_{\tau})$ be a cut-off function wich satisfies $0\leq\chi_{\tau}\leq1$, $\chi_{\tau}(x)=1$ for $x\geq -\frac{1}{4}\tau$ and $\chi_{\tau}(x)=0$ for $x \leq -\frac{1}{2}\tau$. We define $u_{\tau,n}^{\rm qm}:=\chi_{\tau} u_{h,n}$ with $h=\tau^{-1}$. We have $u_{\tau,n}^{\rm qm}\in \dom(\Lunid^{\rm sc}(\tau))$. Since $u_{h,n}$ has exponential decay at infinity we get with $h=\tau^{-1}$: $$\|\Lunid^{\rm sc}(\tau)u_{\tau,n}^{\rm qm}-(2n-1-\tfrac{1}{4\tau^2})u_{\tau,n}^{\rm qm} \|_{L^2(J_{\tau})}=\|Hu_{h,n}-E_{h,n}u_{h,n} \|_{L^2(\R)}+O(h^{\infty}) \ . $$
Using \eqref{E:estimationuhn} and the expansion \eqref{E:DLV} we deduce that for all $n\geq1$ there exists $C_{n}>0$ such that for all $\tau\geq1$:
$$ \|\Lunid^{\rm sc}(\tau)u_{\tau,n}^{\rm qm}-(2n-1-\tfrac{1}{4\tau^2})u_{\tau,n}^{\rm qm}\|_{L^2(J_{\tau})}^2\leq \frac{C_{n}}{\tau^3}  \quad \mbox{and} \quad |\|u_{\tau,n}\|_{L^2(J_{\tau})}-1| \leq \frac{C_{n}}{\tau}  \ . $$
Since $\Lunid^{\rm sc}(\tau)$ is unitary equivalent to $\Lunid(\tau)$ we deduce the asymptotic expansion \eqref{E:DAzetan}.
\end{proof}

\begin{rem}
In the case of a 2D electron gas submitted to the magnetic field $B_{0}$, the group velocity of a quantum particle is given by $-(\mudgN_{n})'(\tau)$ (see \cite[Section II]{PeetRej} for a physical approach or \cite{Yaf08} for a mathematical analysis). Let us notice that unlike for the eigenvalues $\mudgN_{n}(\tau)$ of the de Gennes operator (see \eqref{E:limtaugranddG}), the convergence of $\zeta_{n}(\tau)$ toward the Landau levels $2n-1$ is not exponential. Therefore we should expect different transport properties associated to a 3D electron gas submitted to the magnetic field $\Bcyl$. 
\end{rem}

 If we denote by $m\in \Z$ the magnetic quantum number, the spectral analysis of $\dgG(\tau)$ can be deduced from the analysis of the spectrum of the operators 
 $$  -\partial_{r}^2-\frac{1}{r}\partial_{r}+\frac{m^2}{r^2}+(r-\tau)^2 \ , \quad r>0$$
 acting on $L^2_{r}(\R_{+})$. We denote by $\zeta_{n,m}(\tau)$ the $n$-th eigenvalue of this operator. In this article we have focused on the case $m=0$ and we have denoted by $\zeta_{n}(\tau)=\zeta_{n,0}$. It is proved in \cite[Proposition 3.6]{Yaf08} that all the $\zeta_{n,m}(\tau)$ converge toward the Landau level $2n-1$ for large values of $\tau$. Looking at the proof of Proposition \ref{P:asymptoticzeta}, we can deduce that 
 \begin{equation}
\label{E:convzetanm}
 \forall m\in \Z, \, \forall n\geq1, \quad  \zeta_{n,m}(\tau)\underset{\tau\to+\infty}{=}2n-1+(m^2-\tfrac{1}{4})\frac{1}{\tau^2}+O\left(\frac{1}{\tau^3}\right)  \ . 
 \end{equation}
 Therefore only the eigenvalues of $\dgG(\tau)$ associated to axisymmetric function are below the Landau level for large values of $\tau$. Let us also notice that the second term of this asymptotics does not depend on the energy level $n$.

It is also possible to add an electric perturbation to the magnetic hamiltonians $H_{0}$ and $H_{\Acyl}$. The associated Krein spectral shift function (see \cite{BirYaf92} for an overview on the spectral shift function) will have singularities near the Landau levels $2n-1$ who play the role of ``threesholds" in the spectrum of the operators $H_{0}$ and $H_{\Acyl}$. The asymptotic behavior of the spectral shift function near the threesholds depends among other things on the behavior of the band functions at energies closed to the threesholds (see for example \cite{BriRaiSoc08} for a study of a magnetic hamiltonian on a half-strip). Since the band functions $\mudgN_{n}(\tau)$ and $\zeta_{n}(\tau)$ have different behaviors for large $\tau$, we expect that the spectral shift function associated to perturbations of the hamiltonians $H_{0}$ and $H_{\Acyl}$ will have different singular behaviors when approaching the Landau levels by below. Since all the $(\zeta_{n,m})_{m\neq0}$ converge to the same Landau level $2n-1$ by above (see \eqref{E:convzetanm}), the singular behavior of the SSF when approaching the threesholds by above may also be interesting.
\subsection{Rough upper bound}
Using the estimation \eqref{E:EstYafaev}, it is proved in \cite[Theorem 4.9]{Yaf08} that all the function $\tau\mapsto \zeta_{n}(\tau)$ loose their monotonicity for $\tau>0$ and reach their infimum. We provide an upper bound for the infimum of $\tau\mapsto\zeta_{1}(\tau)$:
\begin{prop}
\label{P:probatteint}
We have 
\begin{equation}
\label{E:upperboundxi0}
 \Xi_{0} \leq \sqrt{4-\pi} 
 \end{equation}
 and there exists $\tau^{*}\in \R$ such that $\Xi_{0}=\zeta_{1}(\tau^{*})$.
\end{prop} 
\begin{proof}
 In order to get an upper bound we use gaussian quasi-modes: for $\gamma>0$ we define $u_{\gamma}(r):=e^{-\gamma r^2}$. Computations yield: 
 $$\frac{\gq_{\tau}(u_{\gamma})}{\|u_{\gamma}\|_{L^2_{r}(\R_{+})}^2}=2\gamma+\frac{1}{2\gamma}+\tau^2-\tau \left(\frac{\pi}{2\gamma}\right)^{1/2} \ .$$
We minimize the right hand side by choosing $\gamma=\frac{\pi}{8\tau^2}$ and we deduce from the min-max principle:
$$\zeta_{1}(\tau) \leq \frac{\pi}{4}\frac{1}{\tau^2}+\frac{4-\pi}{\pi}\tau^2 \ .$$
This upper bound is minimal for $\tau=(\frac{\pi^2}{4(4-\pi)})^{1/4}$ and provides \eqref{E:upperboundxi0} by using \eqref{E:relationfond}. 
\end{proof}

\section{Characterization of the minimum}
\subsection{Comparison between the lowest energies}
In this section we give a new expression of the derivative of the function $\zeta_{n}$. We use it to get a comparison between $\Theta_{0}$ and $\Xi_{0}$.

In order to have a  parameter-independent potential, we perform the translation $\rho=r-\tau$ and we get that $\Lunid(\tau)$ is unitary equivalent to the operator 
$$\LunidT(\tau):=-\partial_{\rho}^2-\frac{1}{\rho+\tau}\partial_{\rho} +\rho^2 \ , \quad \rho>-\tau  $$
acting on $L^2_{\rho+\tau}(I_{\tau})$ with $I_{\tau}:=(-\tau,+\infty)$. The domain of the operator $\LunidT(\tau)$ is deduced from $\dom(\Lunid(\tau))$ using the translation $\rho=r-\tau$. The interval $I_{\tau}$ depends now on the parameter. Usually the technics from \cite{DauHe93} and \cite{DauHe93-II} give a trace formula for the derivative with respect to the boundary of the eigenvalues of such an operator  (see the formula \eqref{F:DH93} for example). However the results of  \cite{DauHe93-II} specific to weighted spaces cannot be applied, indeed the weight $\rho+\tau$ depends on the parameter $\tau$. We prove \`a la ``Bolley-Dauge-Helffer" a formula for the derivative that is of a different kind from \eqref{F:DH93} and \cite[Theorem 1.8]{DauHe93-II}. To our knowledge this formula is independent from the Feynman-Hellmann formula \eqref{E:FH}:
\begin{prop}
\label{P:nouvelleformule}
Let $n\geq1$ and let $z_{n,\tau}$ be a normalized eigenfunction associated to $\zeta_{n}(\tau)$ for the operator $\Lunid(\tau)$. We have 
\begin{equation}
\label{E:nouvellederivee}
 \zeta_{n}'(\tau)= \langle (\dgN_{0}(\tau)-\zeta_{n}(\tau))z_{n,\tau},z_{n,\tau} \rangle_{L^2(\R_{+})} \ .
\end{equation}
\end{prop}
\begin{proof}
We denote by $\hat{z}_{n,\tau}(\rho):=z_{n,\tau}(\rho+\tau)$ a normalized eigenfunction of $\LunidT(\tau)$ associated to $\zeta_{n}(\tau)$. It satisfies 
\begin{equation}
\label{E:eqvp2}
-\hat{z}_{n,\tau}''(\rho)-\frac{\hat{z}_{n,\tau}'(\rho)}{\rho+\tau}+\rho^2\hat{z}_{n,\tau}(\rho)=\zeta_{n}(\tau)\hat{z}_{n,\tau}(\rho)  \ .
\end{equation} 
 For $h>0$ we introduce the quantity 
$$d_{n,\tau}(h):=\left(\zeta_{n}(\tau+h)-\zeta_{n}(\tau) \right)\langle \hat{z}_{n,\tau+h},\hat{z}_{n,\tau} \rangle_{L^2_{\rho+\tau}(I_{\tau})} \ . $$
The analyticity of the eigenpairs $(\zeta_{n}(\tau),\hat{z}_{n,\tau})_{n\in \N}$ is a direct consequence of the simplicity of the eigenvalues (see proposition \ref{P:Fuchs}) and of Kato's theory. Since the $\hat{z}_{n,\tau}$ are normalized in $L^2_{\rho+\tau}(I_{\tau})$ we deduce that 
\begin{equation}
\label{E:formulediabolique}
\lim_{h\to0}\frac{d_{n,\tau}(h)}{h}=\zeta_{n}'(\tau)  \ .
\end{equation}
On the other side using the eigenvalue equation \eqref{E:eqvp2} we get:
\begin{align*}
d_{n,\tau}(h)=&\int_{-\tau}^{+\infty}\big(\zeta_{n}(\tau+h)\hat{z}_{n,\tau+h}(\rho)\hat{z}_{n,\tau}(\rho)-\zeta_{n}(\tau)\hat{z}_{n,\tau+h}(\rho)\hat{z}_{n,\tau}(\rho)\big)(\rho+\tau) \d \rho
\\
=&\int_{-\tau}^{+\infty}\left(-\hat{z}_{n,\tau+h}''(\rho)-\frac{1}{\rho+\tau+h}\hat{z}_{n,\tau+h}'(\rho)+\rho^2\hat{z}_{n,\tau+h}(\rho)\right)\hat{z}_{n,\tau}(\rho)(\rho+\tau) \d \rho
\\
&-\int_{-\tau}^{+\infty}\left(-\hat{z}''_{\tau}(\rho)-\frac{1}{\rho+\tau}\hat{z}_{n,\tau}'(\rho)+\rho^2\hat{z}_{n,\tau}(\rho)\right)\hat{z}_{n,\tau+h}(\rho)(\rho+\tau) \d \rho \ .
\end{align*}
We make integrations by part on the terms with second derivative:
\begin{align*}
d_{\tau}(h)&=\int_{-\tau}^{+\infty}\hat{z}_{n,\tau+h}'(\rho)\big((\rho+\tau)\hat{z}_{n,\tau}'(\rho)+\hat{z}_{n,\tau}(\rho)\big)-\frac{\rho+\tau}{\rho+\tau+h}\hat{z}_{n,\tau+h}'(\rho)\hat{z}_{n,\tau}(\rho) \d \rho
\\
+&\int_{-\tau}^{+\infty}-\hat{z}_{n,\tau}'(\rho)\big((\rho+\tau)\hat{z}_{n,\tau+h}'(\rho)+\hat{z}_{n,\tau+h}(\rho)\big)+\hat{z}_{n,\tau}'(\rho)\hat{z}_{n,\tau+h}(\rho) \d \rho
\\
&=  h \int_{-\tau}^{+\infty}\frac{1}{\rho+\tau+h}\hat{z}_{n,\tau+h}'(\rho)\hat{z}_{n,\tau}(\rho) \d \rho \ .
\end{align*}
Thanks to \eqref{E:formulediabolique} and to the analyticity of the eigenpairs with respect to the parameter we have 
$$\zeta_{n}'(\tau)=\int_{-\tau}^{+\infty}\frac{1}{\rho+\tau}\hat{z}_{n,\tau}'(\rho)\hat{z}_{n,\tau}(\rho) \d \rho \ . $$
Using \eqref{E:eqvp2} we deduce
$$\zeta_{n}'(\tau)=\int_{-\tau}^{+\infty}\left(-\hat{z}_{n,\tau}''(\rho)+\rho^2\hat{z}_{n,\tau}(\rho)-\zeta_{n}(\tau)\hat{z}_{n,\tau}(\rho) \right)\hat{z}_{n,\tau}(\rho)\d \rho \ . $$
We make the translation $\rho=r-\tau$:
$$\zeta_{n}'(\tau)=\int_{\R_{+}}\left(-z_{n,\tau}''(r)+\left((r-\tau)^2z_{n,\tau}(r)-\zeta_{n}(\tau)z_{n,\tau}(r)\right)\right)z_{n,\tau}(r)\d r \ . $$
 Thanks to \eqref{E:vpneumann}, we have $z_{n,\tau}\in \dom(\dgN_{0}(\tau))$ and we deduce \eqref{E:nouvellederivee}.
\end{proof}
This formula is not sufficient to give direct informations on the monotonicity of the functions $\tau\mapsto\zeta_{n}(\tau)$. However it provides the following comparison between the bottom of the spectrum of the 2D hamiltonian $H_{0}$ and the one of the 3D hamiltonian $H_{\Acyl}$:
\begin{theo}
We have 
$$\Theta_{0}<\Xi_{0} \ . $$
\end{theo}
\begin{proof}
Let $\tau^{*}$ be a point such that $\zeta_{1}(\tau^{*})=\Xi_{0}$ (see Proposition \ref{P:probatteint}) and $z_{\tau^{*}}:=z_{1,\tau^{*}}$ an associated eigenfunction such that $\|z_{\tau^{*}}\|_{L^2_{r}(\R_{+})}=1$. We have $\zeta_{1}'(\tau^{*})=0$ and thanks to Proposition \ref{P:nouvelleformule}: 
$$\zeta_{1}(\tau^{*})=\frac{\langle \dgN_{0}({\tau^{*})z_{\tau^{*}},z_{\tau^{*}} \rangle_{L^2(\R_{+})}}}{\|z_{\tau^{*}}\|^2_{L^2}(\R_{+})} \ . $$
We deduce from the min-max principle that $\zeta_{1}(\tau^{*}) \geq \mudgN_{1}(\tau^{*})$. Let us suppose that we have the equality $\zeta_{1}(\tau^{*})=\mudgN_{1}(\tau^{*})$, then $z_{\tau^{*}}$ is a minimizer of the quadratic form associated to $\dgN_{0}(\tau^{*})$ and since $z_{\tau^{*}}$ satisfies the Neumann boundary condition \eqref{E:vpneumann}, it is an eigenfunction of $\dgN_{0}(\tau^{*})$ associated to $\mudgN_{1}(\tau^{*})$ and it satisfies $\dgN_{0}(\tau^{*})z_{\tau^{*}}=\mudgN_{1}(\tau^{*})z_{\tau^{*}}$, that is
$$\forall r>0, \quad -z_{\tau^{*}}''(r)+(r-\tau^{*})^2z_{\tau^{*}}(r)=\mudgN_{1}(\tau^{*})z_{\tau^{*}}(r) \ . $$ Combining this with \eqref{equation-de-Fuchs} we get $z_{\tau^{*}}'=0$ on $\R_{+}$, that is absurd. Therefore we have $\Xi_{0}=\zeta_{1}(\tau^{*})>\mudgN_{1}(\tau^{*})\geq\Theta_{0}$ .
\end{proof}

\subsection{A criterion for the characterization of the minimum}
In \cite[Section 4]{Yaf08}, the author states the question of knowing how many minima has the band function $\tau\mapsto \zeta_{n}(\tau)$. We give here a criterion in order to characterize the critical points of $\zeta_{1}$. Numerical simulations show that this criterion seems to be satisfied. Let us notice that most of the technics presented here can be found in \cite{HePer10}, \cite{HeKo09} and \cite{He10Mong}. In the following we denote by $z_{\tau}$ a normalized eigenfunction of $\Lunid(\tau)$ associated to $\zeta_{1}(\tau)$. Since $\zeta_{1}(\tau)$ is simple, $\tau\mapsto z_{\tau}$ is analytic and we denote by $\dot{z}_{\tau}:=\partial_{\tau}z_{\tau} \ . $
\begin{lem}
\label{L:thespeclambda2}
We have 
$$\forall \tau \in \R , \quad \|\dot{z}_{\tau} \|_{L_{r}^2(\Rp)} \leq \frac{2}{\zeta_{2}(\tau)-\zeta_{1}(\tau)}\|(r-\tau)z_{\tau}\|_{L^2_{r}(\Rp)} \ . $$
\end{lem}
\begin{proof}
We differentiate $\| z_{\tau} \|_{L^2_{r}(\Rp)}^2=1$ with respect to $\tau$ and we get that $\dot{z}_{\tau}$ is orthogonal to $z_{\tau}$ in $L^2_{r}(\Rp)$. We deduce from the min-max principle:
$$(\zeta_{2}(\tau)-\zeta_{1}(\tau))\|\dot{z}_{\tau} \|_{L_{r}^2(\Rp)}^2 \leq \langle (\Lunid(\tau)-\zeta_{1}(\tau))\dot{z}_{\tau},\dot{z}_{\tau}\rangle_{L^2_{r}(\Rp)} . $$
We differentiate $\Lunid(\tau)z_{\tau}=\zeta_{1}(\tau)z_{\tau}$ with respect to $\tau$ and we get $(\Lunid(\tau)-\zeta_{1}(\tau))\dot{z}_{\tau}=-\partial_{\tau}\Lunid(\tau)z_{\tau}$, therefore
$$(\zeta_{2}(\tau)-\zeta_{1}(\tau))\|\dot{z}_{\tau} \|_{L_{r}^2(\Rp)}^2 \leq \langle -\partial_{\tau}\Lunid(\tau)z_{\tau},\dot{z}_{\tau}\rangle_{L^2_{r}(\Rp)} \ . $$
By using Cauchy-Schwarz inequality and the identity $\partial_{\tau}\Lunid(\tau)=-2(r-\tau)$ we deduce the Lemma.
\end{proof}

\begin{lem}[Viriel identity]
\label{L:Viriel}
Let $\tau_{\rm C}$ be a critical point of $\zeta_{1}$. Then we have 
\begin{equation}
\label{E:viriel}
\int_{\R_{+}}r |z_{\tau_{\rm C}}'(r)|^2\d r =\int_{\R_{+}}(r-\tau_{\rm C})^2|z_{\tau_{\rm C}}(r)|^2r \d r= \frac{\zeta_{1}(\tau_{\rm C})}{2} \ .
\end{equation}
\end{lem}
\begin{proof}
  We introduce the scaled operator 
 $$\Lunid(\tau, a):= -a^{-2}\frac{1}{r}\partial_{r}r \partial_{r}+(a r-\tau)^2, \quad a>0 $$
 which is unitary equivalent to $\Lunid(\tau)$. We denote by $z_{\tau}^{a}(r):=z_{\tau}(\frac{r}{a})$ and we have
 $$\forall a>0, \quad \left(\Lunid(\tau, a)-\zeta_{1}(\tau)\right)z_{\tau}^{a}=0 \ . $$ 
 We differentiate this relation with respect to $a$:
 \begin{equation}
\label{E:diffunedeplusc6}
\left(\Lunid(\tau, a)-\zeta_{1}(\tau)\right)\partial_{a}z_{\tau}^{a}+\partial_{a}\Lunid(\tau, a)z_{\tau}^{a}=0  
\end{equation}
 with $$\partial_{a}\Lunid(\tau, a)=2a^{-3}\frac{1}{r}\partial_{r}r\partial_{r}+2r(a r-\tau) \ . $$ We make the scalar product of \eqref{E:diffunedeplusc6} with $z_{\tau}^{a}$ in $L^2_{r}(\R_{+})$ and we take $a=1$: 
\begin{equation}
\label{Id:Viriel1dsing}
\int_{\Rp}\left(-2 |z_{\tau}'(r)|^2+2r(r-\tau)|z_{\tau}(r)|^2\right)r \d r=0 \ .
\end{equation}
Thanks to \eqref{E:FH}, if $\tau_{\rm C}$ is a critical point of $\zeta_{1}$ we have 
$$\int_{\Rp}(r-\tau_{\rm C})|z_{\tau_{\rm C}}(r)|^2 r \d r =0$$
and therefore
\begin{equation}
\label{Id:FHtauC1dsing}
\int_{\Rp}(r-\tau_{\rm C})^2|z_{\tau_{\rm C}}(r)|^2 r \d r =\int_{\Rp}r(r-\tau_{\rm C})|z_{\tau_{\rm C}}(r)|^2 r \d r \ .
\end{equation}
Since $$\forall \tau \in \R, \quad \int_{\Rp} \left(|z_{\tau}'(r)|^2+(r-\tau)^2|z_{\tau}(r)|^2\right)r \d r= \zeta_{1}(\tau), $$
by using \eqref{Id:Viriel1dsing} and \eqref{Id:FHtauC1dsing} we get \eqref{E:viriel}.
\end{proof}
We can now state our criterion: if the spectral gap is large enough in a critical point of $\zeta_{1}$, this critical point is a non-degenerate minimum:
\begin{prop}
\label{P:spectralgap}
Let $\tau_{\rm C}$ be a critical point of $\zeta_{1}$. Then we have 
\begin{equation}
\label{E:specgapimplycritical}
\zeta_{1}''(\tau_{\rm C}) \geq 2\frac{\zeta_{2}(\tau_{\rm C})-3\zeta_{1}(\tau_{\rm C})}{\zeta_{2}(\tau_{\rm C})-\zeta_{1}(\tau_{\rm C})} \ . 
\end{equation}
\end{prop}
\begin{proof}
We first differentiate the Feynman-Hellmann relation \eqref{E:FH} and we get 
$$\forall \tau \in \R, \quad \zeta_{1}''(\tau)=2-4\int_{\R_{+}}(r-\tau)\dot{z}_{\tau}(r)z_{\tau}(r) r \d r \ . $$
The Cauchy-Schwarz inequality and the Lemma \ref{L:thespeclambda2} provide
$$\zeta_{1}''(\tau) \geq 2-\frac{8\|(r-\tau)z_{\tau}\|^2_{L^2_{r}(\R_{+})}}{\zeta_{2}(\tau)-\zeta_{1}(\tau)} \ . $$
If $\tau_{\rm C}$ is a critical point of $\zeta_{1}$, we deduce \eqref{E:specgapimplycritical} from the Lemma \ref{L:Viriel}.
\end{proof}

We know that $\zeta_{2}(0)-3\zeta_{1}(0)=0$ and that $\zeta_{2}-3\zeta_{1}$ goes to 0 for $\tau$ large (see Proposition \ref{P:limite}). Moreover numerical simulations show that $\zeta_{2}-3\zeta_{1}$ seems to be positive on $(0,+\infty)$, see figure \ref{F:2}. We already know that $\zeta_{1}$ is non-increasing on $(-\infty,0)$,  therefore using Proposition \ref{P:spectralgap} we believe that all the critical points of $\zeta_{1}$ are minima. Therefore we are led to make the following: 
\begin{conj}
\label{C:}
The band function $\tau\mapsto\zeta_{1}(\tau)$ has a unique and non-degenerate minimum.
\end{conj}

\begin{rem}
Let us notice that similar conjectures can be found in the litterature: in \cite{PopRay12}, the characterization of the minimum of the band function of a related model problem would bring localization property for the semi-classical Laplacian of a domain with a curved edge. In \cite{Yaf03} the author makes a conjecture on the monotonicity of the derivative of a band function associated to a magnetic hamiltonian in $\R^3$.
\end{rem}

Using the technics from \cite[Section 3]{Yaf03} and \cite[Section 5]{Yaf08}, the conjecture \ref{C:} can bring scattering properties for the hamiltonian $H_{\mathcal{A}}$. Moreover if the conjecture \ref{C:} is true, we will be able to describe the number of eigenstates created under the action of a suitable electric perturbation (see \cite{BriRaiSoc08}). We hope to continue these investigations in the future. 
\paragraph{Acknowledgements} The author is grateful to E. Soccorsi for giving the physical impulsion, for his interest in this work and for precious advices. The author is also grateful to V. Bruneau for helpful discussions.

\newpage
\appendix
\section{Numerical approximations}
\label{S:num}
The numerical approximations described here use the finite element library M\'elina (\cite{Melina++}). We refer to \cite[Subsection 6.2.4]{Popoff} for more simulations and computation details. We denote by $\breve{\zeta_{n}}(\tau)$ a numerical approximation of $\zeta_{n}(\tau)$. The figure \ref{F:1} presents the numerical approximations $\breve{\zeta_{1}}(\tau)$ for $\tau=\frac{k}{100}$ with $0 \leq k \leq 500$. The numerical approximations have a unique minimum $\breve{\Xi}_{0}=0.8630$ and the corresponding minimizing Fourier parameter is $\breve{\tau}^{*}=1.53$ . We have also plotted the constant $\Theta_{0}\approx 0.5901$ according to the computation of \cite{Bo08} and the upper bound $\sqrt{4-\pi}\approx 0.9265$ given by Proposition \ref{P:probatteint}.

The figure \ref{F:2} presents the numerical approximation $\breve{\zeta_{2}}(\tau)-3\breve{\zeta_{1}}(\tau)$ for $\tau=\frac{k}{100}$ with $0 \leq k \leq 500$. These quantities are positive for $\tau>0$, therefore we think that $\zeta_{2}(\tau)-3\zeta_{1}(\tau)$ is positive for all $\tau>0$. Using Proposition \ref{P:spectralgap}, we believe that the conjecture \ref{C:} is true.
\begin{figure}[ht]
\begin{center}
\includegraphics[keepaspectratio=true,width=15cm]{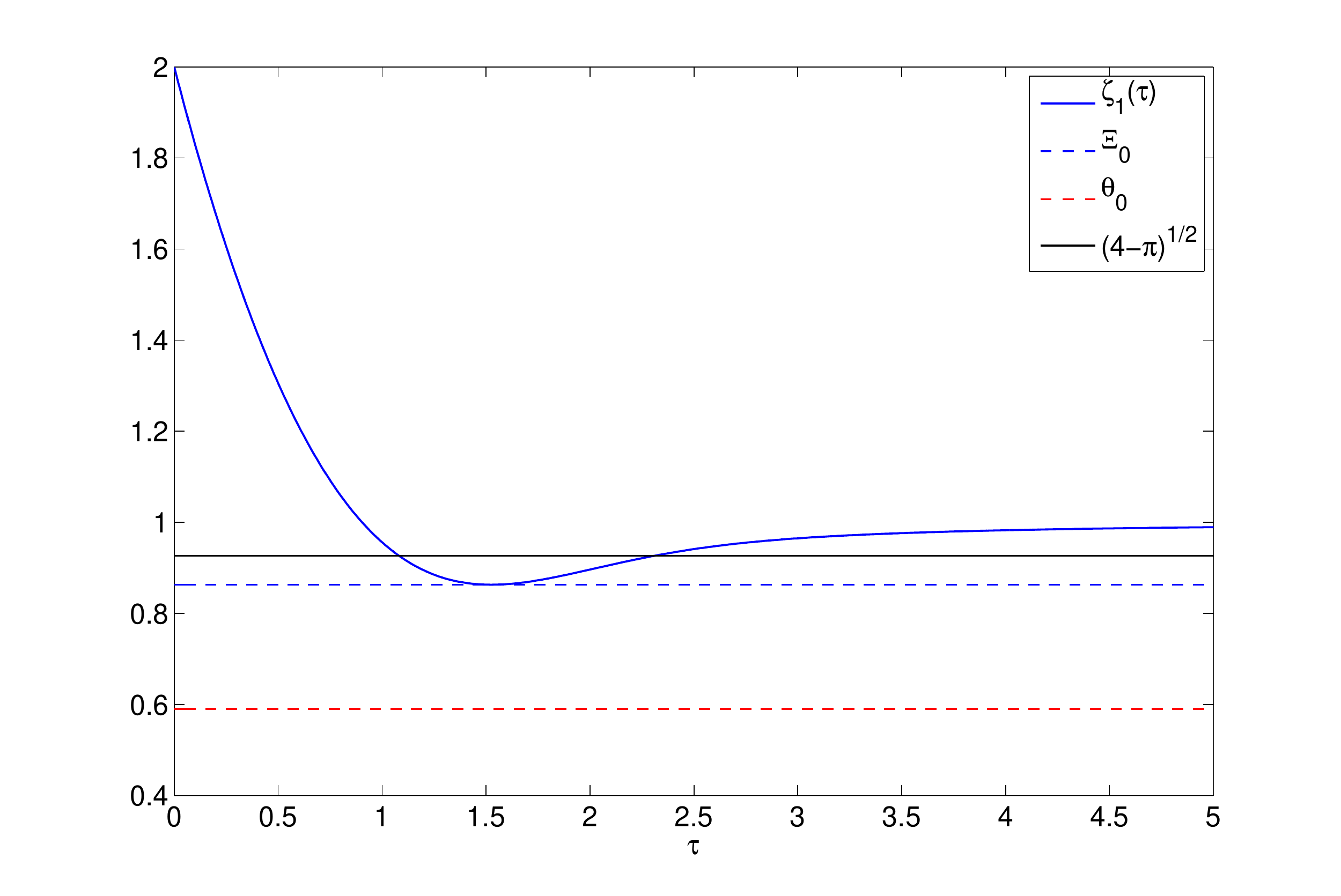}
\caption{The numerical approximations $\breve{\zeta_{1}}(\tau)$ for $\tau=\frac{k}{100}$ with $0 \leq k \leq 500$ compared to the constant $\Theta_{0}$ and the upper bound $\sqrt{4-\pi}$.}
\label{F:1}
\end{center}
\end{figure}

\begin{figure}[ht!]
\begin{center}
\includegraphics[keepaspectratio=true,width=15cm]{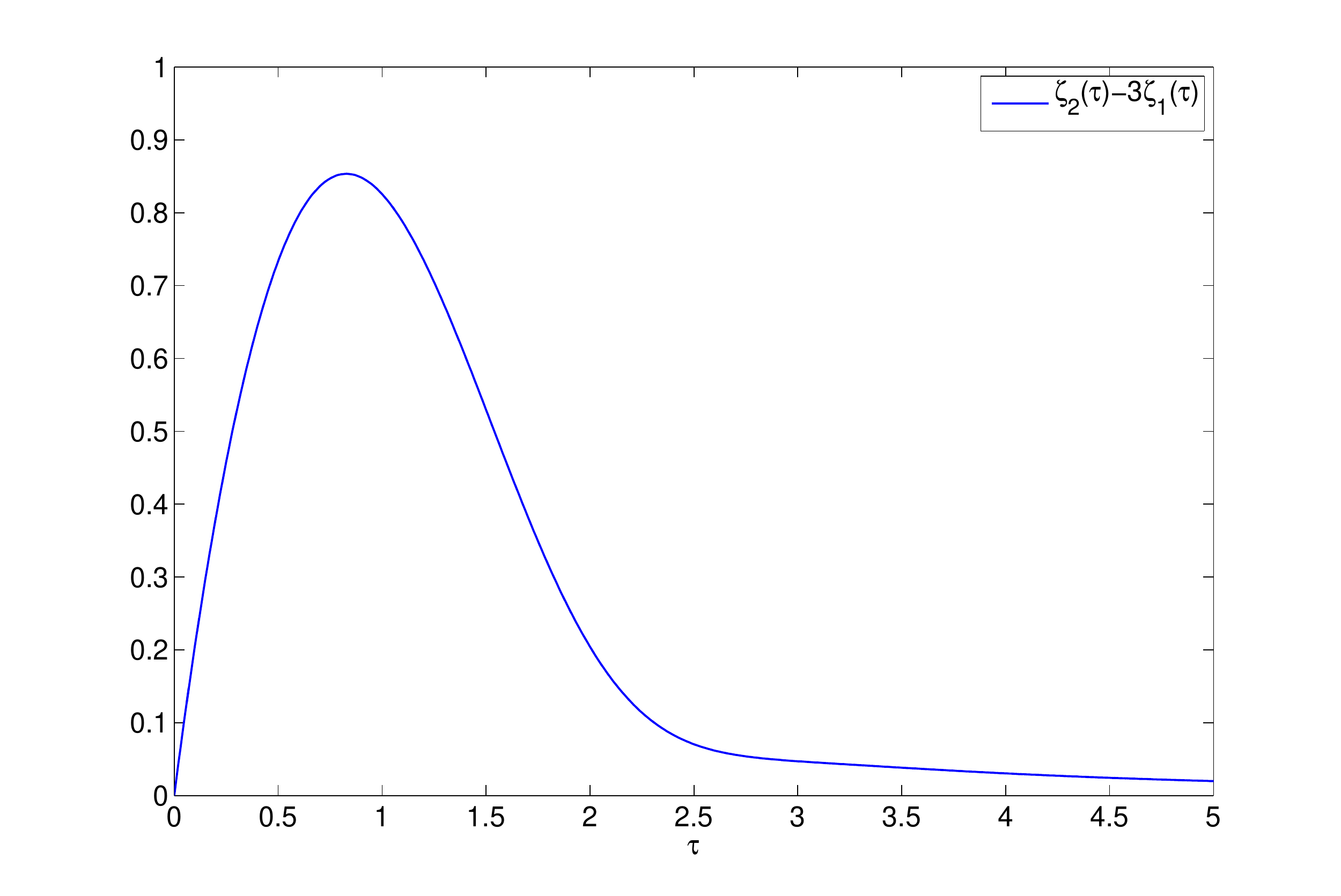}
\caption{The numerical approximations $\breve{\zeta_{2}}(\tau)-3\breve{\zeta_{1}}(\tau)$ for $\tau=\frac{k}{100}$ with $0 \leq k \leq 500$.}
\label{F:2}
\end{center}
\end{figure}


\newpage
\bibliographystyle{mnachrn}
\bibliography{biblio}
\end{document}